%% file: arxiv_disconnected.tex
\documentclass{article}
\usepackage[utf8]{inputenc}
\usepackage[margin=1in]{geometry}
\usepackage{amsmath} 
\usepackage[ruled]{algorithm2e}
\input{macro/jared_macro}

\title{\bf Algebraic Proofs of Path Disconnectedness \\ using Time-Dependent Barrier Functions}

\author{
Didier Henrion\thanks{LAAS-CNRS, Universit\'e de Toulouse, CNRS, Toulouse, France;
Faculty of Electrical Engineering, Czech Technical University in Prague, Czech Republic. (henrion@laas.fr)} \and Jared Miller\thanks{Automatic Control Laboratory (IfA), Department of Information Technology and Electrical Engineering (D-ITET), ETH Z\"{u}rich, Physikstrasse 3, 8092, Z\"{u}rich, Switzerland (e-mail: jarmiller@control.ee.ethz.ch)} \and
Mohab Safey El Din\thanks{
Sorbonne Universit\'e,
LIP6, UMR CNRS 7606,
PolSys Team, Paris, France. (mohab.Safey@lip6.fr)}}

\date{\today}

\begin{document}

\maketitle

\input{sections/abstract}
\input{sections/intro}
\input{sections/summary}
\input{sections/preliminaries}

\input{sections/path_connected}

\input{sections/path_disconnected}

\input{sections/box_model}

\input{sections/sos_disconnected}

\input{sections/examples}
\input{sections/conclusion}
\input{sections/acknowledgements}

\appendix
\input{sections/app_strict}
\input{sections/app_poly_approx}

\bibliographystyle{IEEEtran}
\bibliography{references}

\end{document}

%% file: macro/jared_macro.tex
\usepackage[printonlyused]{acronym}
\usepackage{xcolor}
\usepackage{amsmath}
\usepackage{amssymb}
\usepackage{mathtools}
\usepackage{subcaption}
\usepackage{xcolor}
\usepackage{hyperref}
\hypersetup{colorlinks=true, linkcolor=black, citecolor=black}
\usepackage{mathrsfs}
\usepackage{cite}
\usepackage{adjustbox}
\usepackage{float}

\usepackage{bm}

\DeclareMathOperator*{\find}{find}
 
\newcommand{\R}{\mathbb{R}} 
\newcommand{\K}{\mathbb{K}} 
 
\newcommand{\N}{\mathbb{N}}

\newcommand{\Lie}{\mathcal{L}} 
\newcommand{\ks}{\mathcal{K}}

\newcommand{\A}{\mathcal{A}}

\newcommand{\psd}{\mathbb{S}}

\newcommand{\xs}{\mathcal{X}}
\newcommand{\ys}{\mathcal{Y}}

\DeclarePairedDelimiter{\abs}{\lvert}{\rvert}
\DeclarePairedDelimiter{\norm}{\lVert}{\rVert}

\DeclarePairedDelimiterX{\inp}[2]{\langle}{\rangle}{#1, #2}
\DeclarePairedDelimiter{\supp}{\textrm{supp}(}{)}
\DeclarePairedDelimiter{\Mp}{\mathcal{M}_+(}{)}

\usepackage{amsthm}
\newtheorem{thm}{Theorem}[section]
\newtheorem{lem}[thm]{Lemma}
\newtheorem{prop}[thm]{Proposition}
\newtheorem{cor}{Corollary}

\newtheorem{rem}{Remark}

%% file: sections/abstract.tex
\begin{abstract}
    
Two subsets of a given set are path-disconnected if they lie in different connected components of the larger set. Verification of path-disconnectedness is essential in proving the infeasibility of motion planning and trajectory optimization algorithms. We formulate path-disconnectedness as the infeasibility of a single-integrator control task to move between an initial set and a target set in a sufficiently long time horizon. This control-infeasibility task is certified through the generation of a time-dependent barrier function that separates the initial and final sets. The existence of  a time-dependent barrier function is a necessary and sufficient condition for path-disconnectedness under compactness conditions.
Numerically, the search for a polynomial barrier function is formulated using the moment-sum-of-squares hierarchy of semidefinite programs. The barrier function proves path-disconnectedness at a sufficiently large polynomial degree. The computational complexity of these semidefinite programs can be reduced by elimination of the control variables. Disconnectedness proofs are synthesized for example systems.
\end{abstract}



%% file: sections/intro.tex
\section{Introduction}
\label{sec:introduction}


Let $X_0$ and $X_1$ be compact sets included in a compact set $X$ of $\R^n$. The sets $X_0$ and $X_1$ are path-connected inside $X$ if there exists a pair of points $x_0 \in X_0, \ x_1 \in X_1$ and a continuous 
function $x: [0, 1] \rightarrow X$ such that $x(0) = x_0$ and $x(1) = x_1$. The sets $X_0$ and $X_1$ are path-disconnected in $X$ if there does not exist such a function $x$. Equivalently, the sets $X_0$ and $X_1$ are path-disconnected if they lie in different connected components of $X$.

Deciding whether $X_0$ and $X_1$ are path-connected in $X$ is a core problem for
motion planning.  When $X_0$, $X_1$ and $X$ are \emph{semi-algebraic sets} (i.e.
defined by disjunctions of conjunctions of polynomial inequalities with real
coefficients), this problem boils down to the \emph{roadmap problem}. It
consists in computing a certificate of path-connectedness, hence a curve, that
would connect, in $X$, one point in $X_0$ to one point in $X_1$. Such a problem
has attracted much attention since the pioneering work of Canny~\cite{Canny88}
who showed that such certificates can be computed in time polynomial in the
maximum degree of the polynomial constraints and exponential in $n^2$. This
research track is still active with exciting complexity improvements, reducing
the dependancy on $n$, see e.g. \cite{SaSc17, PreSaSc24, presasc-preprint}.
When $X_0$ and $X_1$ cannot be path-connected in $X$, these algorithms will just
return a curve in $X$ that does not connect $X_0$ to $X_1$. The user is expected
to trust the algorithm and its implementation. Hence note that these algorithms
do not provide a certificate that can be checked a posteriori in the
disconnected case. 

{While computer algebra algorithms provide an exact solution to the connectedness query, they are computationally expensive. In engineering applications, numerical probabilistic and sampling-based motion planning algorithms are generally preferred for their more favorable running time, see e.g. \cite[Chapter 5]{Lav06}. If $X_0$ and $X_1$ are path-connected, these algorithms generate a connecting curve with probability one. If $X_0$ and $X_1$ are not path-connected, these algorithms do not terminate, and no certificate of disconnectedness is returned.}


The work in
\cite{delanoue2006using} uses interval analysis and iterative refinement to
determine if a set is path-connected or path-disconnected. However, this method
requires full-dimensional sets in order to provide certification, and is unable
to handle equality-constraint-generated submanifolds.
{Similarly to other motion planning algorithms, interval arithmetic algorithms do not provide certificates of disconnectedness that can be verified independently by a third party.}

The key idea behind our approach consists of interpreting the search of a path connecting $X_0$ and $X_1$ as \iac{OCP} under state and control constraints. Following \cite{lewis1980relaxation}, this \ac{OCP} is then reformulated as an infinite-dimensional \ac{LP} in cones of positive measures called occupation measures. Disconnectedness then amounts to infeasibility of the measure \ac{LP}, and this can be certified by a Farkas vector solving a dual linear problem in cones of positive continuous functions \cite{barvinok2002convex}. 
This Farkas vector may be interpreted as a time-dependent barrier function which is strictly positive on $X_0$, non-positive on $X_1$, and increases along all possible controlled trajectories, as studied previously in \cite{prajna2004safety,prajna2005necessity}.
Under our compactness assumptions, the existence of a barrier function is necessary and sufficient for proving path-disconnectedness, though time-dependence of the barrier function is required due to the failure of the Slater condition (i.e. existence of an interior point in the dual LP) which is assumed in \cite{prajna2005necessity} for time-independent barrier functions. 

The infinite-dimensional primal-dual LPs can be solved numerically by a hierarchy of convex Moment-\ac{SOS} \acp{SDP} of increasing size ruled by the degree of the polynomial barrier function \cite{henrion2008nonlinear, lasserre2009moments}.






%% file: sections/summary.tex
This paper is laid out as follows:
Section \ref{sec:preliminaries} reviews preliminaries such as notation, the infinite-dimensional Farkas lemma, barrier functions, and occupation measures. Section \ref{sec:path_connected} poses the path-connectedness program as a feasibility \ac{LP} in occupation measures. Section \ref{sec:path_disconnected} forms a functional \ac{LP} that certifies path-disconnectedness through the existence of a time-dependent barrier function. Section \ref{sec:box} reduces the complexity of finding this barrier function by eliminating the control variables. Section \ref{sec:sos} applies the moment-\ac{SOS} hierarchy to find these barrier functions. Section \ref{sec:examples} demonstrates our algorithm on proving path-disconnected of example systems. Section \ref{sec:conclusion} concludes the paper. Appendix \ref{app:strict} certifies that the barrier function program may be expressed using strict inequalities without introduction of conservatism. Appendix \ref{app:poly_approx} proves that the time-dependent barrier function certificates may be polynomials.

%% file: sections/preliminaries.tex
\section{Preliminaries}
\label{sec:preliminaries}

\subsection{Acronyms/Initialisms}

\input{sections/acronym}

\subsection{Notation}

$\R[x]$ is the ring of polynomials with vector indeterminates $x$, and $\R_{\leq d}[x]$ is the vector space of polynomials with total degree at most $d$. A monomial in 
$\R[x]$ may be expressed in multi-index notation as $x^\alpha = \prod_i x_i^{\alpha_i}$ for an exponent $\alpha \in \N^n$. The degree of a monomial is $\abs{\alpha} = \sum_i \alpha_i$, and the degree of a polynomial is the maximum degree of all of its monomials.
A basic semialgebraic is a set formed by a countable number of polynomial inequality constraints of bounded degree. 

Let $X \subset \R^n$ be a compact set. The set of continuous functions on $X$ is $C(X)$, and $C_+(X)$ is the subcone of nonnegative continuous functions. An indicator function $I_A(x)$ for $A \subseteq X$ is a function that takes on the value $1$ is $x \in A$ and the value $0$ if $x \not\in A$.
$C^1(X)$ is the set of functions that possess continuous first derivatives. The topological dual to a space $\xs$ is denoted $\xs^*$. Given two elements $f \in \xs$ and $\mu \in \xs^*$, the duality pairing $\inp{f}{\mu}$ is a bilinear map from $\xs \times \xs^*$ to $\R$. An example is $\xs=C(X)$, with dual $\xs^*={\mathcal M}(X)$, the set of signed Borel measures on $X$. The duality pairing between elements $f \in C(X)$ and $\mu \in {\mathcal M}(X)$ is then $\inp{f}{\mu} = \int_X f(x) d \mu$ by Lebesgue integration. The cone $C_+(X)$ of non-negative continuous functions and the cone $ \Mp{X}$ of non-negative Borel measures are dual cones with respect to this pairing. Given a cone $\ks \in \xs$,  its dual cone $\ks^*$ is the set of continuous linear functionals that are non-negative on $\ks$, i.e. $\ks^* = \{v \in \xs^* \mid \inp{v}{x} \geq 0, \ \forall x \in \ks\}$.

The mass (volume) of a measure is $\mu(X) = \inp{1}{\mu} = \int_X d \mu$. Measures with mass one are called probability measures.
The support of a measure $S=\supp{\mu}$ is the smallest closed subset $S \subseteq X$ such that $\mu(X \setminus S) = 0$. 
The Dirac delta $\delta_{x}$ is a probability measure such that $\inp{f}{\delta_{x}} = f(x), \ \forall f \in C(X)$. 

For a linear operator $\A$, the adjoint operator $\A^\dagger$ is defined as the unique operator such that $\inp{\A f}{\mu} = \inp{f} {\A^\dagger \mu}$ for all choices of $f \in C(X), \ \mu \in \Mp{X}$.

Given a domain $\Omega$, the $C^0$ norm of a function $f$ is $\norm{f}_{C^0(\Omega)} = \sup_{x \in \Omega} \abs{f(x)}$, and its $C^1$ norm is  $\norm{f}_{C^1(\Omega)} = \norm{f}_{C^0(\Omega)} + \sum_{i=1} \norm{\partial_i f}_{C^0(\Omega)}$.


\subsection{Conic Feasbility and Alternatives}

Let $\xs, \ \ys$ be locally convex topological spaces, and $\ks$ be a closed convex cone in $\xs$.
The following result provides conditions for the existence of a feasible point in a conic program.

\begin{lem}[Farkas' Lemma \cite{CK77}]\label{lem:farkas}
Let $\A: \ys \rightarrow \xs$ be a continuous linear map with  adjoint $\A^\dagger: \xs^* \rightarrow \ys^*$. Let  $\bm{b} \in \ys^*$.
Assume that the set $\A^\dagger(\ks^*)$ is weak-star closed. The following two feasibility programs are strong alternatives, i.e.  exactly one of the two problems has a solution:

\begin{minipage}[t]{0.4\textwidth}
\begin{subequations}
\label{eq:farkas_dual}
\begin{align}
   \find_x \quad  &\A^\dagger(\bm{x}) = \bm{b} \\
    &\bm{x} \in \ks^*
\end{align}
\end{subequations}
\end{minipage}
\begin{minipage}[t]{0.4\textwidth}
\begin{subequations}
\label{eq:farkas_primal}
\begin{align}
   \find_{y} \quad  &\inp{\bm{y}}{\bm{b}} = -1 \\
    &\bm{y} \in \ys \\
    & \A(\bm{y}) \in \ks
\end{align}
\end{subequations}
\end{minipage}
\end{lem}
The existence of a vector $\bm{y}$ satisfying \eqref{eq:farkas_primal} certifies that \eqref{eq:farkas_dual} does not possess a solution (is infeasible). An $x$ satisfying \eqref{eq:farkas_dual} implies that the infimal value of $\inp{\bm{y}}{\bm{b}}$ is non-negative for all $\bm{y} \in \ys, \ \A(\bm{y}) \in \ks$.

\subsection{Barrier Functions}

Let $X \subset \R^n$ and $U \subset \R^m$ be compact sets, and let $f : \R^n\times \R^m \to \R^n$ be a $C^1$ function. The initial set $X_0 \subset X$ is {safe} with respect to the unsafe set $X_1 \subset X$ if for every initial state $x(0) \in X_0$ there is a control $u : [0,\infty) \to U$ such that the solution $x : [0,\infty) \to X$ of the differential equation $\dot{x}=f(x,u)$ satisfies $x(t) \notin X_1$ for all $t\geq 0$. Safety may be proven by means of barrier functions.

\begin{thm}[Theorem 1 of \cite{prajna2005necessity}]
    \label{thm:barrier}

A sufficient condition for the initial set $X_0$ to be {safe} with respect to the terminal set $X_1$ is that there exists a (barrier) function $v(x)$ satisfying
\cite{prajna2004safety, prajna2006barrier}:
\begin{subequations}
\label{eq:barrier}
   \begin{align}
   \find_v \quad     &v(x) \leq 0 & &  \forall x \in X_1 \label{eq:barrier_xu}\\
       &v(x) > 0 & & \forall x \in X_0 \label{eq:barrier_x0}\\
       &f(x, u) \cdot \nabla_x v(x) \geq 0 & & \forall x \in X, \ u \in U\label{eq:barrier_lie}
    \end{align}
\end{subequations}
where $\nabla_x$ denotes the gradient w.r.t. $x$.
This sufficient condition is necessary if a Slater (i.e. interior point) condition holds: there exists a $C^1$ function $\tilde{v}$ such that $\forall x \in X: f(x, u) \cdot \nabla_x \tilde{v}(x) > 0$.
\end{thm}
The barrier function begins positive on $X_0$ \eqref{eq:barrier_x0} and increases along all trajectories \eqref{eq:barrier_lie}. It is therefore not possible for trajectories to visit $X_1$ where the barrier function is non-positive \eqref{eq:barrier_xu}. The existence of a $v$ that solves \eqref{eq:barrier} is sufficient to certify safety of trajectories with respect to $X_1$. Barrier functions are in general non-unique.

The conditions in \eqref{eq:barrier} may be relaxed while still returning a barrier certificate for safety. The constraint in \eqref{eq:barrier_lie} ensures that the level sets of $v$ are invariant. Condition \eqref{eq:barrier_lie} may be modified to find the existence of a class-$\mathcal{K}$ function $\kappa$ (i.e. $\kappa(0) = 0$, $\kappa$ is nondecreasing) such that \cite{ames2019control}
\begin{align}
    &f(x, u) \cdot \nabla_x v(x) + \kappa(v(x)) \geq 0 & & \forall x \in X, \ u \in U.\label{eq:barrier_lie_kappa}
\end{align}
Constraint \eqref{eq:barrier_lie_kappa} ensures that the zero-level set of $v$ is invariant, while allowing $v$ to fall within the safe set where it is strictly positive.
Constraint \eqref{eq:barrier_xu} may also be slackened to $v(x) \leq 0 \ \forall x \in \partial X_1$ with no loss of generality if $X_0 \cap X_1 = \varnothing$. 


\subsection{Occupation Measures}
\label{sec:occ_measures}
Let $x_0 \in X_0$ be an initial condition of the dynamical system $\dot{x} = f(t, x)$ for a time range $t \in [0, T]$. The trajectory starting at $x_0$ is denoted by $x(t \mid x_0)$. Given subsets $A \subset [0, T], \ B \subset X$, the occupation measure returns the total amount of time the trajectory $x(t \mid x_0)$ spends in the region $A \times B$:

\begin{equation}
    \label{eq:occ_measure_single}
\mu(A \times B \mid x_0) = \int_0^T I_{A \times B}(t, x(t \mid x_0)) dt
\end{equation}
where $I_S$ denotes the function equal to one on $S$ and zero outside.
If $\mu_0 \in \Mp{X_0}$ is a measure over initial conditions, the average occupation measure w.r.t. $\mu_0$ is
\begin{equation}
\label{eq:occ_measure}
    \mu(A \times B) = \int_X \mu(A \times B \mid x_0) d \mu_0(x_0).
\end{equation}
The distribution of states at time $t=T$ found by tracking trajectories starting from $\mu_0$ is
\begin{equation}
\label{eq:final_measure}
    \mu_T(B) = \int_X I_B(x (T \mid x_0)) d \mu_0(x_0).
\end{equation}
The initial measure $\mu_0$, average occupation measure $\mu$, and final measure $\mu_T$ are linked by the continuity equation, also called the Liouville equation.
Let us use the symbol $\Lie_f: C^1([0, T] \times X) \rightarrow C([0, T] \times X)$ to refer to the Lie derivative operator along the vector field $f(x,u)$:
\begin{align}
    v \mapsto \Lie_f v(t, x) = \partial_t v(t, x) + f(x,u) \cdot \nabla_x v(t, x). \label{eq:lie_u}
\end{align}
The Liouville equation takes the following forms:
\begin{subequations}
\label{eq:liou}
\begin{align}
        \inp{v(T, x)}{\mu_T} &= \inp{v(0, x)}{\mu_0} + \inp{\partial_t v(t,x) + f(t,  x) \cdot \nabla_x v(t,x)}{\mu} \label{eq:liou_int} \\ 
        \delta_T \otimes\mu_T  &=  \delta_0 \otimes\mu_0 + (\partial_t + f \cdot \nabla_x)^\dagger \mu. \label{eq:liou_final}
\end{align}
\end{subequations}
The linear equation on measures \eqref{eq:liou_final} is equivalent (in a distributional sense) to the weak integral form \eqref{eq:liou_int} which holds for all test functions $v(t,x) \in C^1([0, T] \times X)$. Two consequences of Liouville's theorem is that $\mu_0(X_0) = \mu_T(X)$ (with $v = 1$) and $\mu([0, T] \times X) = T$ (with $v = t$). 

Control action can be incorporated into the occupation measure formulation. Let $U$ be a compact set of plausible controls at each moment in time, where the null control $u=0$ is an interior point of $U$. Two such examples are the unit ball $U = \{u \in \R^n \mid \norm{u}_2^2  \leq 1\}$ and the unit box $U =  [-1, 1]^n$.
A control occupation measure can be defined for any subset $C \subset U$:
\begin{equation}
    \label{eq:occ_measure_w}
    \mu(A \times B \times C) = \quad \int_{[0, T] \times X \times U} I_{A \times B \times C}((t, x(t), u(t)) \mid x_0) d \mu_0(x_0). 
\end{equation}
The Liouville equation in \eqref{eq:liou} can be extended to control-occupation measures as
\begin{subequations}
\label{eq:liou_u}
\begin{align}
        \inp{v(T, x)}{\mu_T} &= \inp{v(0, x)}{\mu_0} + \inp{\partial_t v(t,x) + f(t, x, u) \cdot \nabla_x v(t,x)}{\mu} \label{eq:liou_u_int} \\ 
        \delta_T \otimes\mu_T  &=  \delta_0 \otimes\mu_0 + \pi^{tx}_\# (\partial_t + f \cdot \nabla_x)^\dagger \mu \label{eq:liou_u_final}
\end{align}
\end{subequations}
where the pushforward of the projection measure $\pi^{tx}_\# \mu$ marginalizes out the control $u$, and yields an occupation measure $\nu$ in $(t, x)$. More explicitly, $\pi^{tx} : (t,x,u) \mapsto (t,x)$ is the projection map on the $(t,x)$ coordinates, and the push-forward measure $d\nu(t,x):=\pi^{tx}_\# d\mu(t,x,u)$ is such that
$\int_{A\times B\times C} v(t,x)d\mu(t,x,u) = \int_{A\times B\times C} v(t,x)\pi^{tx}_\# d\nu(t,x)$ for all test functions $v(t,x) \in C^1([0,T]\times X)$ and all subsets $A \subset [0,T]$, $B \subset X$, $C \subset U$.




%% file: sections/acronym.tex
\begin{acronym}[SOS]
\acro{BSA}{Basic Semialgebraic}

\acro{LMI}{Linear Matrix Inequality}
\acroplural{LMI}[LMIs]{Linear Matrix Inequalities}
\acroindefinite{LMI}{an}{a}

\acro{LP}{Linear Program}
\acroindefinite{LP}{an}{a}

\acro{OCP}{Optimal Control Problem}
\acroindefinite{OCP}{an}{an}


\acro{PD}{Positive Definite}

\acro{PSD}{Positive Semidefinite}

\acro{SDP}{Semidefinite Program}
\acroindefinite{SDP}{an}{a}

\acro{SOS}{Sum-of-Squares}
\acroindefinite{SOS}{an}{a}

\acro{WSOS}{Weighted Sum-of-Squares}


\end{acronym}

%% file: sections/path_connected.tex
\section{Path Connectedness}
\label{sec:path_connected}
This paper poses path-disconnectedness as the infeasibility of a measure program derived from the framework of \acp{OCP}. Such infeasibility can be proven by the necessary and sufficient existence of time-dependent barrier functions through the Farkas lemma.

This section begins by formulating \acp{LP} that certify path-connectedness. The subsequent section poses alternative \acp{LP} through the Farkas lemma to certify path-disconnectedness.

As in calculus of variations, we treat the coordinates $x \in X \subset \R^n$ as states of the single-integrator dynamical system $\dot{x} = u$ for a control input $u(\cdot): [0, \infty) \rightarrow U$.

\subsection{Assumptions}
The following assumptions hold throughout this paper.
\begin{itemize}
    \item[A1] The sets $X_0, X_1, X$ are compact.
    \item[A2] The input set $U$ is a convex, compact, full-dimensional set containing the origin.
\end{itemize}

\subsection{Time Horizons for Path-Connectedness}

The path-connectedness problem may be formulated as \iac{OCP}.
Assuming that two given points $x_0, x_1$ lie within the same connected component of $X$, the following \ac{OCP} returns the Euclidean geodesic distance, or path length, between $x_0$ and $x_1$ in $X$:
$$\begin{subequations}
\label{eq:geodesic_l2}
\begin{array}{rcclll}
\tau_X(x_0,x_1)& := &\inf_{x(\cdot), \tau} & \tau \\
&&\mathrm{s.t.}\: & x(0) = x_0, & x(\tau) \in x_1 \\
&&&x(t) \in X, & \dot{x}(t) \in U := \{u \in \R^n \mid \norm{u}_2 \leq 1\},  & \forall t \in [0, \tau]. \label{eq:geodesic_l2_input}
\end{array}
\end{subequations}$$
Note that this OCP can also be equivalently formulated as follows:
$$\begin{array}{rcclll}
\tau_X(x_0,x_1)& = &\inf_{x(\cdot), u(\cdot)}& \int_0^1 \|u(t)\|_2 dt \\
&&\mathrm{s.t.}\: & x(0) = x_0, & x(1) \in x_1 \\
&&&x(t) \in X, & \dot{x}(t) = u(t),  & \forall t \in [0, \tau]. \label{eq:geodesic_l2_input}
\end{array}$$
Choosing other control sets $U$ subject to assumption A2  results in other finite values of the geodesic distance for other metrics.

Assume that the set $X$ may be decomposed into a union of connected components $X = \cup_{i=1}^{N_c}X^{i}$.
For a fixed input set $U$ under A2, a horizon $T^i$ may be generated as the maximal time to connect any pair of points in $X^i$:
$$\begin{subequations}
\label{eq:geodesic_connected}
\begin{array}{rcllll}
    T^i & = &\sup_{x_0, x_1 \in X^i} & \inf_{x(\cdot), \tau} \tau \\
&&&\mathrm{s.t.}    &x(0) = x_0, \ x(\tau) \in x_1 \\
&&&&x(t) \in X & \forall t \in [0, \tau] \\
 & & & &\dot{x}(t) \in U & \forall t \in [0, \tau]. \\
\label{eq:geodesic_connected_input}
\end{array}
\end{subequations}$$
The maximal time required to connect any pair of points in the same connected component $X$ is
\begin{align}
    T_{X} = \max_{i=1..N_c} T^i. \label{eq:time_connected}
\end{align} 

\subsection{Upper-Bounds on Time Horizons}

The maximal time \eqref{eq:time_connected} may be computationally difficult to find. In certain cases, upper-bounds $T \geq T_X$ may be computed. As an example, consider the case where $X$ may be expressed as a finite union of $N_b$ 2-dimensional boxes $X = \cup_{j=1}^{N_b} [a^j_1, a^j_2] \times [b^j_1, b^j_2]$. Under the Euclidean scenario where $U = \{u \in \R^2 \mid \norm{u}_2 \leq 1\}$, the true connectivity time horizon $T_X$ is upper-bounded by the finite quantity 
\begin{equation*}
    T_X \leq \sum_{j=1}^{N_b} \sqrt{(a_2^j - a_1^j)^2 + (b_2^j-b_1^j)^2} = T.
\end{equation*}
Another upper-bound $T \geq T_X$ may be insantiated if $X$ is a connected set generated by a single polynomial inequality constraint:
\begin{thm}[Theorem 2.1 of \cite{d2003bounds}]
Let $X := \{x \in \R^n \mid g(x) \geq 0\} \subseteq B^n$ (unit Euclidean ball), where $g$ is a given polynomial of degree $d$ with $n, d \geq 2$. Then the maximum Euclidean geodesic distance between any two points of $X$ satisfies
\begin{subequations}
\label{eq:kurdyka_bound_all}
\begin{align}
     \sup_{x_0, x_1 \in X} \tau_X(x_0, x_1) \leq 4 \Gamma\left(\frac{1}{2}\right) \Gamma\left(\frac{n+2}{2}\right) \Gamma\left(\frac{n+1}{2}\right)^{-1} d(4d-5)^{n-1} \label{eq:kurdyka_bound}\\ \intertext{where $\Gamma$ is the Euler Gamma function.}
\end{align}
    \end{subequations}
\end{thm}

\begin{proof}
Applying Theorem 2.1 of \cite{d2003bounds} gives an upper bound on the maximum Euclidean geodesic distance on the real algebraic variety $X_g=\{(x,y)\in \R^{n+1} : g(x)=y^2\}$. The set $X_g$ is compact because $X$ is compact.
\end{proof}

\subsection{Path Feasibility}

This subsection poses \iac{LP} in measures that can provide a proof of path-connectedness.


Define $T_s$ as the minimal amount of time required to connect any pair of points in $X_0 \times X_1$ via the single-integrator dynamics:
\begin{subequations}
    \label{eq:min_time_orig}
\begin{align}
    T_s = & \inf_{x_0, x_1, u}  \tau &\\
     & x(0) = x_0 \in X_0 \\
     & x(\tau) = x_1 \in X_1 &  \\
     & x(t) \in X & t \in [0, \tau] \\
     & \dot{x}(t) = u(t), \ u(t) \in U & t \in [0, \tau]
\end{align}
\end{subequations}

\begin{prop}
    The points $(x_0, x_1)$ are path-connected if $T_s \leq T_X < \infty$. The points $(x_0, x_1)$ are path-disconnected if $T_s=\infty$, which implies that $T_s > T \geq T_X$.
\end{prop}

Occupation measures may be used to determine if there exists a path between $X_0$ and $X_1$. 




\begin{prop}
There exists a path between $X_0$ and $X_1$ if and only if the following program has a solution:
\begin{subequations}
\label{eq:connect_meas}
\begin{align}
     \find_{\mu_0, \mu_p, \mu} \quad & \pi^{tx}_\#\Lie_u^\dagger \mu + \delta_0 \otimes \mu_0 = \delta_T \otimes \mu_T \label{eq:connect_liou}\\
      &\mu_0(X_0) = 1 \label{eq:connect_prob}\\
      &\mu_0 \in  \Mp{X_0}, \quad \mu_T \in \Mp{X_1}\\
      &\mu \in \Mp{[0, T] \times X \times U}.\label{eq:connect_occ}
\end{align}
\end{subequations}
\end{prop}
\begin{proof}
See \cite[Lemma 3]{HK14} where it is proven with the help of Ambrosio's superposition theorem \cite{ambrosio2008transport} that to any triplet of measures $(\mu_0,\mu_T,\mu)$ solving LP \eqref{eq:connect_meas}, there exists a family of
absolutely continuous admissible trajectories
for OCP \eqref{eq:min_time_orig} starting from the support of $\mu_0$ such that the occupation measure and the terminal measure generated by this family
of trajectories are equal to $\mu$ and $\mu_T$ respectively.

The measures $\mu_0$  and $\mu_1$ are probability distributions over $X_0$ and $X_1$ by constraint \eqref{eq:connect_prob}. If \eqref{eq:connect_meas} has a solution and $\mu_0$ and $\mu_1$ are Dirac measures, then a path exists between $x_0 = \supp{\mu_0}$ and $x_1 = \supp{\mu_1}$. Feasibility of constraint \eqref{eq:connect_liou} implies that $\mu$ is supported on the graph $(t, x(t), u(t))$ such that $x(t)$ is the path $x_0 \rightarrow x_1$ and $u(t)$ is the control action (velocities) achieving this path \cite{lasserre2008nonlinear}. By condition \eqref{eq:connect_occ}, the trajectory $x(t)$ connecting $x_0 \rightarrow x_1$ spends all of its time in $X$. 
For general measures $\mu_0, \ \mu_1$, feasibility  of \eqref{eq:connect_meas}  implies that there is a way to connect some set of points in $X_0$ ($\supp{\mu_0}$) to another set of points in $X_1$ ($\supp{\mu_1}$) where the path stays within $X$. The sets $X_0$  and $X_1$ are therefore path-connected within $X$. 

\end{proof}

Note that terminal time $T$ is such that the path between $X_0$ and $X_1$ remains in $X$ within a time horizon of $T$. It can be used as a decision variable in LP \eqref{eq:connect_meas}, since it appears linearly as the mass of the occupation measure $\mu$.

%% file: sections/path_disconnected.tex
\section{Path Disconnectedness}
\label{sec:path_disconnected}
This section formulates time-dependent barrier functions to certify infeasibility of \eqref{eq:connect_meas} within a time horizon $T \geq T_X$.

\subsection{Path Infeasibility}

If there does not exist a solution to \eqref{eq:connect_meas}, then $X_0$ and $X_1$ are path-disconnected (given that $T \geq T_X$). Farkas' Lemma \ref{lem:farkas} may be used to certify nonexistence of  measures $\mu_0, \ \mu_1, \ \mu$ that satisfy \eqref{eq:connect_meas}. 

\begin{thm}
    LP \eqref{eq:connect_meas} is infeasible if and only if the following LP is feasible:
\begin{subequations}
\label{eq:altern_u}
\begin{align}
     \find_v \quad &v(0, x) \geq 1 & & \forall x \in X_0 \label{eq:altern_u_init}\\
    &v(T, x) \leq 0 & & \forall x \in X_1 \label{eq:altern_u_term}\\
    &\Lie_u v(t, x) \geq 0 & & \forall (t, x, u) \in [0, T] \times X \times U \label{eq:altern_u_lie}\\
    &v(t, x) \in C^1([0, T] \times X). 
\end{align}
\end{subequations}
\end{thm}
\begin{proof}
Problem \eqref{eq:connect_meas} has the form of the conic duality program in \eqref{eq:farkas_dual}, with the following parameters:
\begin{subequations}
\begin{align}
    \bm{x} & = [\mu_0, \ \mu_1, \ \mu] \\
    \ks^* &= \Mp{X_0} \times \Mp{X_1} \times \Mp{[0, T] \times X \times U} \\
    \bm{b} &= [0, \ 1] \\
    \A(\bm{x}) & = [-\delta_T \otimes \mu_T + \delta_0 \otimes \mu_0 + \pi^{tx}_\# \Lie_u^\dagger \mu, \ \mu_0(X_0)].
\end{align}
\end{subequations}

The alternative program from \eqref{eq:farkas_primal} is derived as
\begin{subequations}
\begin{align}
    \bm{y} & = [v, \ \gamma] \\
    \inp{\bm{y}}{\bm{b}} &= \gamma = -1 \\
    \ys &= C^1([0, T] \times X) \times \R \\ 
    \ks &= C_+(X_0) \times C_+(X_1) \times C_+([0, T] \times X \times U) \\
    \A^\dagger(\bm{y}) & = [v(0, x) + \gamma, -v(T, x), \Lie_u v].
\end{align}
\end{subequations}

This application of Farkas' Lemma \ref{lem:farkas} completes the proof.
\end{proof}

The function $v(t,x)$ starts out positive \eqref{eq:altern_u_init} and increases along all controlled trajectories corresponding to admissible inputs $u \in U$ \eqref{eq:altern_u_lie}. Because $v \leq 0$ in $X_1$ \eqref{eq:altern_u_term}, no trajectory starting from $X_0$ will reach $X_1$ for any control $u$. A function $v$ satisfying \eqref{eq:altern_u} is a certification of infeasibility for LP \eqref{eq:connect_meas} between times $t \in [0, T]$. The level set $v(t,x)=0$ is a time-dependent barrier function separating $X_0$ (negative)  and $X_1$ (nonnegative). 

\begin{prop}
In case $X_0$ and $X_1$ are single points $(x_0, x_1)$, LP \eqref{eq:altern_u} has a simpler form:
\begin{subequations}
\label{eq:altern_u_pt}
\begin{align}
    \find_v \quad &v(0, x_0) \geq 1 \qquad v(T, x_1) \leq 0& \label{eq:altern_u_pt_end}\\
    &\Lie_u v(t, x) \leq 0 & \forall (t, x, u) \in [0, T] \times X \times U \label{eq:altern_u_pt_traj}\\
    &v(t, x) \in C^1([0, T] \times X).
\end{align}
\end{subequations}
\end{prop}
The inequalities in \eqref{eq:altern_u_pt_end} may be replaced with equalities such as $v(0, x_0) = -1, \ v(T, x_1) = 1$ after appropriate scaling. These (in)equalities are point evaluations of the function $v$.

\begin{cor}    
    LP \eqref{eq:altern_u} may be restricted to strict inequalities without introduction of conservatism:
    \begin{subequations}
\label{eq:altern_u_strict}
\begin{align}
     \find_v \quad &v(0, x) \geq 1 & & \forall x \in X_0 \label{eq:altern_u_strict_init}\\
    &v(T, x) \leq 0 & & \forall x \in X_1 \label{eq:altern_u_strict_term}\\
    &\Lie_u v(t, x) \geq 0 & & \forall (t, x, u) \in [0, T] \times X \times U \label{eq:altern_strict_u_lie}\\
    &v(t, x) \in C^1([0, T] \times X). 
\end{align}
\end{subequations}
\end{cor}
\begin{proof}
See Appendix \ref{app:strict}.
\end{proof}

\subsection{Failure of Slater's Condition}



Time-dependent barriers functions are required for the Farkas alternative condition in \eqref{eq:altern_u} to certify path-disconnectedness. If the function $v(t,x)$ was chosen to be time-independent as $v(x)$, then $\partial_t v(x) = 0$. The time-independent alternative program (with $v=B$ from \eqref{eq:barrier}) would be:
\begin{subequations}
\label{eq:altern_u_indep}
\begin{align}
     \find_v \quad &v(x) \geq 1 & & \forall x \in X_0 \label{eq:altern_u_indep_init}\\
    &v(x) \leq 0 & & \forall x \in X_1 \label{eq:altern_u_indep_term}\\
    &u \cdot \nabla_x v(x) \geq 0 & & \forall (x, u) \in  X \times U \label{eq:altern_u_indep_traj}\\
    &v(t, x) \in C^1(X).
\end{align}
\end{subequations}
The inner product in condition \eqref{eq:altern_u_indep_traj} must be nonnegative for all possible values of $u \in U$ (choices of signs), given that $0$ is an interior point of $U$ by assumption A2. The only way this nonnegativity can occur is for $\nabla_x v(x) = \mathbf{0}$, which implies that $v(x) = c$ is constant for some $c$.
Constraint \eqref{eq:altern_u_indep_init} would require that $c \geq 1$ while \eqref{eq:altern_u_indep_term} imposes that $c \leq 0$.
This is a contradiction, because a $c \in \R$ cannot simultaneously satisfy $c \geq 1$ and $c \leq 0$.
Slater's condition from Theorem 1 of \cite{prajna2005necessity} is violated because there does not exist a time independent $v(x)$ satisfying \eqref{eq:altern_u_indep}.

%% file: sections/box_model.tex
\section{Box Model} 
\label{sec:box}
The inequality constraint in \eqref{eq:altern_u_0_strict_lie} is posed with $2n+1$ variables $(t,x,u)$. In the case where the control set $U$ is chosen to be a box under A2 and A3, the $u$ variables may be eliminated through the methods of \cite{majumdar2014convex,  korda2014controller, miller2023robustcounterpart}. This elimination does not change the feasibility properties of finding time-dependent barrier functions, and in fact reduces the computational expense of solving \acp{SDP} arising from the barrier \ac{LP}. 
In this section we modify assumption A2 as follows:
\begin{itemize}
    \item[A2'] The control set is the unit box $U = [-1,1]^n$.
\end{itemize}

\begin{rem}
    The maximal time horizon $T_X$ from \eqref{eq:time_connected} is computed with respect to the Euclidean  ball controller. Because the unit Euclidean ball is included in the box $[-1, 1]^n$, a time-dependent barrier function from \eqref{eq:altern_u} at $T \geq T_X$ with $U=[-1, 1]^n$ will certify path-disconnectedness in time $T$.
\end{rem}




The strict Lie constraint in \eqref{eq:altern_u_0_strict_lie} under the unit-box-control restriction may be expressed as
\begin{align}
\partial_t v(t, x) +  u \cdot \nabla_x v(t, x) > 0 & & \forall (t, x, u) \in [0, T] \times X \times [-1, 1]^n.\label{eq:altern_u_0_strict_lie_box}
\end{align}
Multiplier functions $\zeta^\pm$ can be introduced to eliminate the control variables $u \in U$.
\begin{thm}
\label{thm:altern_u_box}
Given a function $v$ that satisfies \eqref{eq:altern_u_0_strict_lie_box} under assumptions A1 and A4, there exists continuous functions $\zeta^\pm$ such that
\begin{subequations}
\label{eq:altern_u_box_con}
    \begin{align}
   \find_{v, \zeta^\pm} \quad &\partial_t v(t,x) -  \sum_{i=1}^n \zeta^+_i(t, x) - \zeta^-_i(t, x)  > 0  & & \forall (t, x, u) \in [0, T] \times X\label{eq:altern_u_box_lie} \\
    & \zeta^+_i(t, x) - \zeta^-_i(t, x) = \partial_{x_i} v(t, x) & & \forall i=1..n \\
    & \zeta^+_i,\  \zeta^-_i \in C_+([0, T] \times X) & & \forall i=1..n. \label{eq:altern_u_box_con_zeta}
\end{align}
\end{subequations}
\end{thm}
\begin{proof}
    Equivalence of \eqref{eq:altern_u_0_strict_lie_box} and \eqref{eq:altern_u_box_con} holds (with possibly discontinuous $\zeta$) by Theorem 3.2 of \cite{miller2023robustcounterpart}, given that the description of the control set $U$ is $(t, x)$-independent, the dynamics $\dot{x}=u$ are Lipschitz, and all sets from A1 are compact with $T$ finite. Continuity of the $\zeta$ functions are ensured by Theorem 3.3 of \cite{miller2023robustcounterpart}, due to the strictness of the inequality in the Lie constraint \eqref{eq:altern_u_0_strict_lie_box}.
\end{proof}


%% file: sections/sos_disconnected.tex
\section{SDP-based Certification}
\label{sec:sos}

This section will present \ac{SOS}-based finite truncations to the infinite-dimensional disconnectedness \acp{LP} in \eqref{eq:altern_u} (with and without the box-constraint in  \eqref{eq:altern_u_box_con}).

\subsection{Preliminaries of SOS methods}

A polynomial $p \in \R[x]$ is \ac{SOS} if there exists a set of polynomials $\{q_j(x)\}_{j=1}^N$ such that $p(x) = \sum_{j=1}^N q_j(x)^2$. The cone of \ac{SOS} polynomials is $\Sigma[x] \subset \R[x],$ and the subset of \ac{SOS} polynomials of degree up to $2d$ is $\Sigma_{2d}[x] \subset \R_{\leq 2d}[x]$. 
To each \ac{SOS} polynomial $p \in \Sigma[x]$, there exists a polynomial vector $z(x) \in \R[x]^s$, and \iac{PSD} \textit{Gram} matrix $Q \in \psd_+^{s}$ such that $p(x) = z(x)^T Q z(x)$. Given a matrix decomposition $Q = R^T R$, the square-sum factors may be expressed as $q(x) = R z(x)$. Determination if a degree-2$d$ polynomial is \ac{SOS} may be solved through \iac{SDP}. When $z(x)$ is the vector of all monomials from degrees $1..d$, the Gram matrix $Q$ has size $\binom{n+d}{d}$. For a given $p(x)$, computing a $Q$ such that $p(x) = z(x)^T Q z(x)$ in the monomial basis requires $\binom{n+2d}{2d}$ {coefficient matching} equality constraints.

A basic semialgebraic set is described by a finite number of bounded-degree polynomial equality and inequality constraints. An example of such a set is $\K = \{x \mid g_i(x) \geq 0, h_j(x) = 0, \forall i = 1..N_g, j=1..N_h\}$. A sufficient condition for a polynomial $p$ to be nonnegative over $\K$ is if there exists multipliers $(\sigma, \mu)$ such that
\begin{equation}
\label{eq:putinar}
    \begin{aligned}
        & p(x) = \sigma_0(x) + \textstyle \sum_i {\sigma_i(x)g_i(x)} + \sum_j {\mu_j(x) h_j(x)}\\
        & \sigma_0(x) \in \Sigma[x] \qquad \sigma_i(x) \in \Sigma[x] \qquad \mu_j \in \R[x].
    \end{aligned}
\end{equation}
When the degree of each term $\sigma_0$, $\sigma_i g_i$, $\mu_j h_j$ in \eqref{eq:putinar} does note exceed $2d$,
the cone of polynomials $p$ that have such a representation is denoted $\Sigma_{2d}[\K]$. It is called the truncated quadratic module, or sometimes the \ac{WSOS} cone, associated to the representation of $\K$. This notation is convenient but possibly misleading, as the cone $\Sigma_{2d}[\K]$ depends explicitly on the polynomials $g$ and $h$ used to describe $\K$. If the same set $\K$ is described by different polynomials, then the cone $\Sigma_{2d}[\K]$ may change.
If $\K$ is bounded and there exists a constant $R> 0 $ such that $R-\norm{x}_2^2 \in \Sigma_{2d}[\K]$ for sufficient large degree $d$, the representation of $\K$ is said to be Archimedean \cite{putinar1993compact}. 

\subsection{SOS Tightenings of Path-Disconnectedness}

An assumption is required to utilize the polynomial optimization framework for path-disconnectedness.
\begin{itemize}
    \item[A3] All the sets $X_0, X_1, X, U$ have Archimedean representations.
\end{itemize}

The box $U=[-1, 1]^n$ from assumption A2 satisfies the Archimedean requirement in assumption A3.

The \ac{SOS} tightening of degree $2d$ of the disconnectedness program in \eqref{eq:altern_u} is
        \begin{subequations}
\label{eq:altern_u_sos}
    \begin{align}
  \find_{v} \quad  & v(0, x) -1 \in \Sigma_{2d}[X_0] \label{eq:altern_u_sos_1}\\
    & -v(T, x) \in \Sigma_{2d}[X_1] \label{eq:altern_u_sos_2}\\
    &\partial_t v(t, x) +  u \cdot \nabla_x v(t, x) \in \Sigma_{2d}[[0, T] \times X \times U]  \label{eq:altern_u_sos_3}\\
    & v(t, x) \in \R_{2d}[t, x].
\end{align}
\end{subequations}

\begin{thm}
\label{thm:sos_disconnected}
    Assuming that $(X_0, X_1)$ are path-disconnected in $X$ and that A1-A3
hold, program \eqref{eq:altern_u_sos} finds a time-dependent barrier function $v$ as the degree $d$ tends to $\infty$.
\end{thm}
\begin{proof}    
  When $d\to\infty$, the \ac{WSOS} constraints on the right-hand-side in \eqref{eq:altern_u_sos_1}-\eqref{eq:altern_u_sos_3}describe all possible positive polynomials over their respective sets due to the Archimedean assumption A3 \cite{putinar1993compact}. Theorem \ref{thm:poly_v} certifies that a polynomial $v$ satisfying the strict inequalities \eqref{eq:altern_u_0_strict} exists under the given assumptions. Therefore, choosing $d$ sufficiently high will recover the polynomial $v$ that certifies path-disconnectedness.

\end{proof}

Applying a box-input substitution from Theorem \ref{thm:altern_u_box} to the \ac{SOS} program in \eqref{eq:altern_u_sos} leads to 
        \begin{subequations}
\label{eq:altern_u_box_sos}
    \begin{align}
  \find_{v, \zeta^\pm} \quad  & v(0, x) -1 \in \Sigma_{2d}[X_0] \\
    & -v(T, x) \in \Sigma_{2d}[X_1] \\
    &\partial_t v(t, x) -  \sum_{i=1}^n \zeta^+_i(t, x) - \zeta^-_i(t, x) \in \Sigma_{2d}[[0, T] \times X]  \label{eq:altern_u_box_sos_lie}\\
    & \zeta^+_i(t, x) - \zeta^-_i(t, x) - \partial_{x_i} v(t, x) = 0 & & \forall i=1..n \label{eq:altern_u_box_sos_eq} \\
    & \zeta^+_i,\  \zeta^-_i \in \Sigma_{2d}[[0, T] \times X] & & \forall i=1..n \label{eq:altern_u_box_sos_zeta}\\
    & v(t, x) \in \R_{2d}[t, x].
\end{align}
\end{subequations}

\begin{thm}
\label{thm:sos_disconnected_box}
    Under the same conditions as in Theorem \ref{thm:sos_disconnected}, program \eqref{eq:altern_u_sos} finds a time-dependent barrier function as the degree $d$ tends to $\infty$. 
\end{thm}
\begin{proof}    
Theorem \ref{thm:sos_disconnected} certifies that there exists a polynomial $v$ that can be detected through \ac{SOS} methods. Theorem \ref{eq:poly_zeta} certifies that polynomial multipliers $\zeta$ exist under a strict Lie derivative inequality constraint. Therefore, the $u$-eliminated \ac{SOS} path-disconnected program \eqref{eq:altern_u_box_sos} will also converge a feasible certificate as the degree $d$ approaches $\infty$.
\end{proof}

\begin{rem}\label{rem:union}
    The programs in \eqref{eq:altern_u_sos} and \eqref{eq:altern_u_box_sos} may be generalized to cases where $X_0$, $X_1$, and $X$ are the unions of Archimedean basic semialgebraic sets. For example, if $X_0$ is the union $X_0 = \cup_{j=1}^{N_0} X_0^j$, then constraint \eqref{eq:altern_u_sos_1} could be formulated as,
   \begin{align}
        & v(0, x) -1 \in \Sigma_d[X_0^j]  & & \forall j \in 1..N_0.
   \end{align}
\end{rem}

\subsection{Computational Complexity}

The computational burden of solving \eqref{eq:altern_u_sos} and \eqref{eq:altern_u_box_sos} mostly depends on the size of the largest \ac{PSD} Gram matrix in any constraint of degree $2d$. Constraint \eqref{eq:altern_u_sos_3} involves the $2n+1$ variables $(t, x, u)$, leading to a Gram matrix of maximal size $\binom{2n+1+d}{d}$. The input-eliminated constraint in \eqref{eq:altern_u_box_sos_lie} has $n+1$ variables, leading to a Gram matrix of size $\binom{n+1+d}{d}$. In a case where $n=3$ and $d=6$, the maximal size \ac{PSD} matrix falls from $\binom{2*3+1+6}{6}=1716$ to $\binom{3+1+6}{6}=210$ after eliminating $u$. The per-iteration complexity of solving an \ac{SOS} \ac{SDP} derived from \eqref{eq:altern_u_box_sos} using an interior point method scales in a jointly polynomial manner as in  $O((n+1)^{6d})$ and $O(d^{4n})$ \cite{lasserre2009moments, miller2022eiv_short}.

\begin{rem}
    A direct solution to \eqref{eq:altern_u_box_sos} may require a very high degree $d$ in order to produce a certificate of path-disconnectedness. The work in \cite{cibulka2021spatio} presents a time-space partitioning scheme that decomposes $[0, T] \times X$ into a set of cells (e.g. hypercubes). Each cell $c$ has an individual barrier function $v_c$, and compatibility rules are imposed to form splines (in space) and falling transitions (in time). This partitioning scheme yields a piecewise-defined time-dependent barrier function, which could yield certificates at lower polynomial degree $d$ at the expense of additional polynomial equality and inequality constraints.
\end{rem}

%% file: sections/examples.tex
\section{Experiments}

\label{sec:examples}

All experiments in this section were done with input set $U = [-1, 1]^n$ (A4). Matlab 2021a code to generate all examples is available at \url{https://github.com/jarmill/set_connected}. Dependencies include Yalmip \cite{Lofberg2004} and Mosek \cite{mosek92}.

\subsection{Univariate Example}

The first example involves a univariate and disjoint set $X = [a_1, a_2] \cup [b_1, b_2]$.

A degree $d=4$ time-dependent barrier function $v(t, x)$ is produced by \eqref{eq:altern_u_sos} for a choice of parameters $X = [0, 0.4] \cup [0.8, 1],$ $X_0 = 0.2,$ $X_1 = 0.9,$ $T=1$. The union $X$ is modeled as $X = \{x(0.4-x) \geq 0 \} \cup \{(x-0.8)(1-x) \geq 0\}$, recall Remark \ref{rem:union}. This certificate is visualized in Figure \ref{fig:univariate_1box}, in which the black contour is the $v=0$ level set separating the circle $X_0$ and the star $X_1$. The gray walls are the boundaries at $x=0.4$ and $x=0.8$. This univariate example can also be interpreted to prove disconnectedness of an annulus, in which $x$ may be interpreted as the radius.

\begin{figure}[h]
    \centering
    \includegraphics[width=0.7\linewidth]{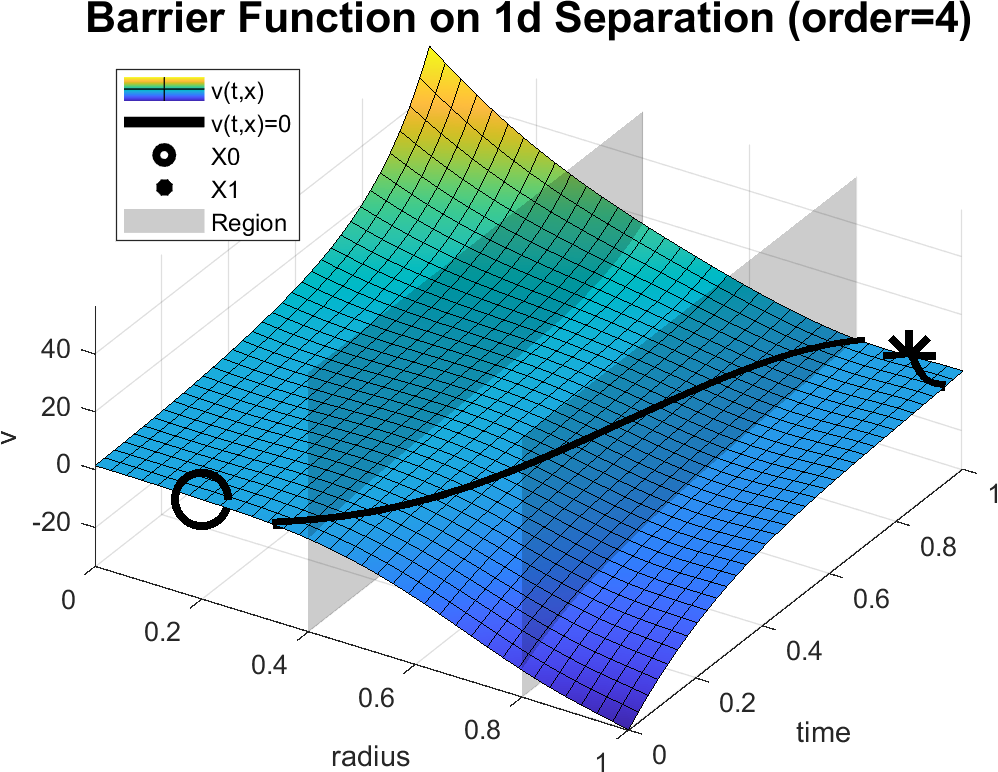}
    \caption{Univariate barrier certificate}
    \label{fig:univariate_1box}
\end{figure}

\subsection{Elliptic Curve}

The second example involves the interior of a noncompact elliptic curve:
\begin{align}
    X = \{x \in \R^n \mid x_2^2 - x_3^2 - 0.8x_1 + 0.05 \geq 0\}. \label{eq:elliptic_curve_set}
\end{align}

With $X_0 = [-0.4; 0.1], X_1 = [0.8; 0.4], T=\sqrt{2}$, the degree $2d=6$ \ac{SOS} tightening of \eqref{eq:altern_u_box_sos} yields a time-dependent barrier function certifying separation. Contours of this certificate are displayed in Figures \ref{fig:ecc-2d} and \ref{fig:ecc-3d}. Note how the circle $X_0$ and star $X_1$ are separated by the zero-level set of $v(t, x)$. Even though the set \eqref{eq:elliptic_curve_set} violates assumption A1 and A3 (the noncompact set \eqref{eq:elliptic_curve_set} is noncompact and non-Archimedean), feasibility of \eqref{eq:altern_u_box_sos} yielded a sufficient certificate of path-disconnectedness.

\begin{figure}[h]
    \centering
    \includegraphics[width=0.6\linewidth]{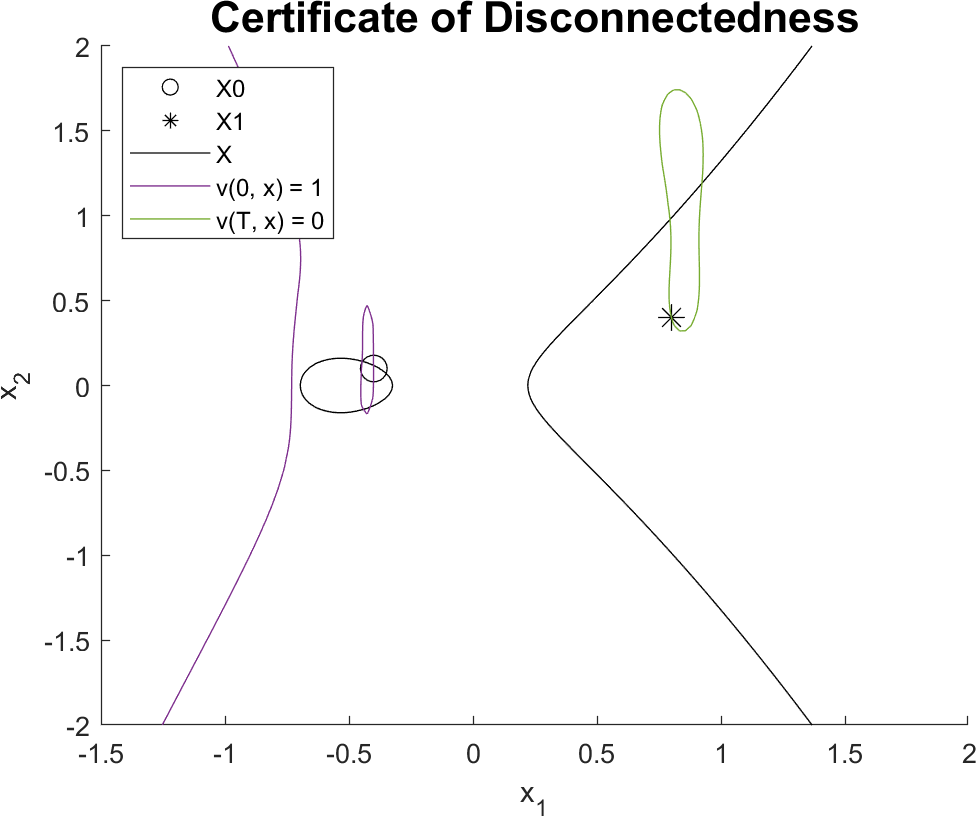}
    \caption{Elliptic curve separation contour}
    \label{fig:ecc-2d}
\end{figure}

Figure \ref{fig:ecc-3d} displays the $v=0$ level-set surface as a function of time $t$ and state $x$.

\begin{figure}[!h]
    \centering
    \includegraphics[width=0.6\linewidth]{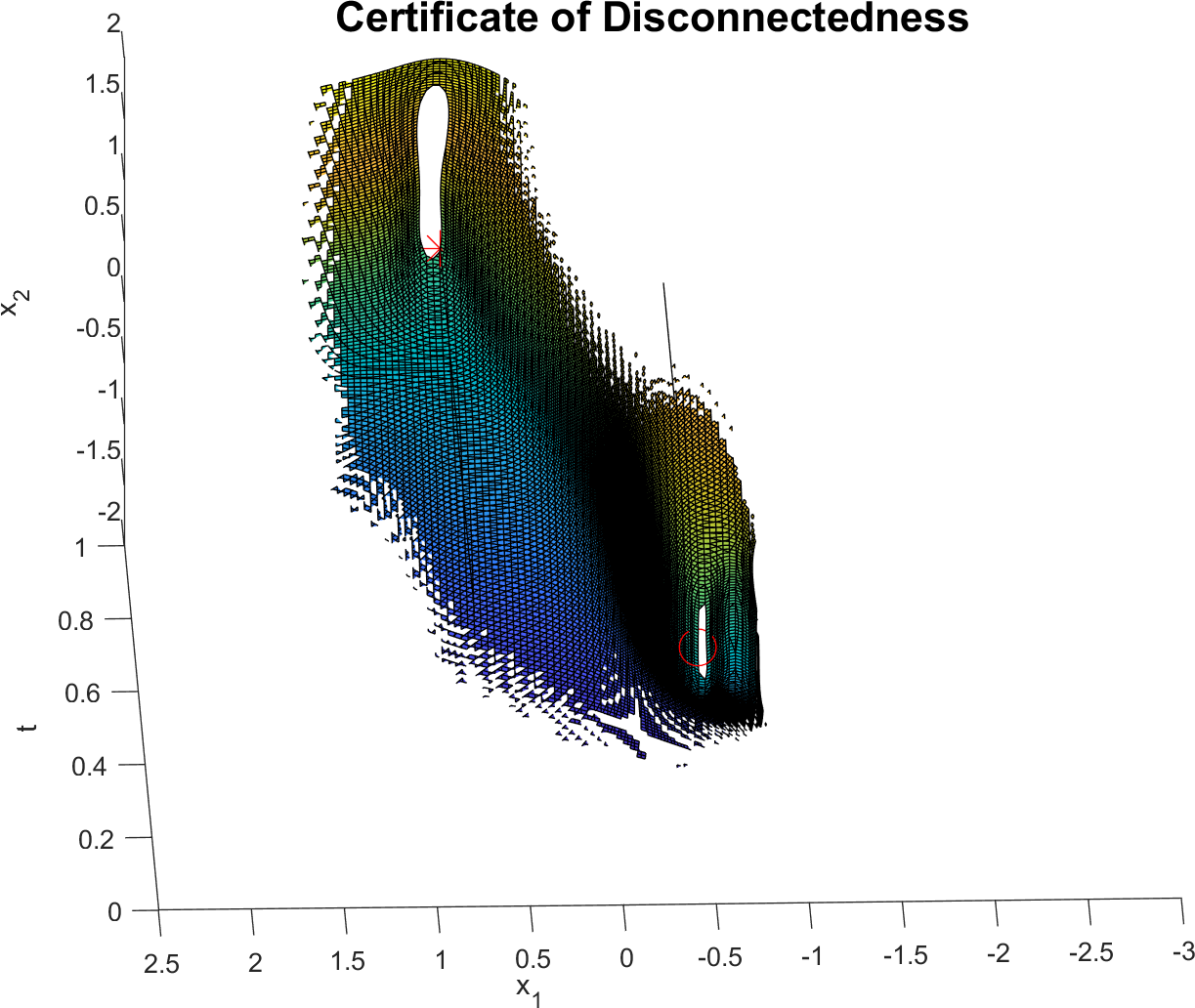}
    \caption{Elliptic curve separation contour (3d)}
    \label{fig:ecc-3d}
\end{figure}

\subsection{Other Demonstrations}

Tables \ref{tab:sys_2d} and \ref{tab:sys_3d} report successful execution of our set-disconnectedness certification scheme (degree $2d$ \ac{SOS} tightening of \eqref{eq:altern_u_box_sos}) within the set $\bar{X} = [-1, 1]^n$. Each set $X$ is defined by one additional inequality constraint: $ X = [-1, 1]^n \cap \{x \mid g(x) \geq 0\}$. 

\begin{table}[!h]
    \centering
    \begin{tabular}{l|ccccc}
        Description & $T$ & order & $g(x)$ & $X_0$ & $X_1$ \\ \hline
        horizontal cut & 2 & 1 & $0.01-(x_1)^2$ & [-0.33; -0.55] & [0.55; 0.275] \\ 
        slanted cut & $\sqrt{2}$ & 2 & $0.01-(x_1+x_2)^2$ & [-0.55; 0] & [0.55; 0.275] \\ 
        arc cut & $\sqrt{2}$ & 2  & $-((x_1-1)^2 + x_2^2 - 0.6^2)(0.4-(x_1-1)^2 - x_2^2)$ & [1; 0] & [-1; 1] \\
         hyperelliptic curve & 1 & 1 & $-y^2 - x(x+1)(x-0.8)(x+0.5)(x-0.5)$ & [-0.85; -0.2] & [0.2; 0.1]
    \end{tabular}
    \caption{Disconnectedness certificates of 2-dimensional systems}
    \label{tab:sys_2d}
\end{table}

\begin{table}[!h]
    \centering
    \begin{tabular}{l|ccccc}
        Description & $T$ & order & $g(x)$ & $X_0$ & $X_1$ \\ \hline
        expanded elliptic curve & 1 & 4 & $x^3 - 0.8 x + 0.05 - (2y^2+z^2/2)$ & [0;0;0] & [0.9; 0.2; 0.2] \\
        hollow ball  & $\sqrt{3}$ & 3 & $(\norm{x}_2^2 - 0.6^2)(0.4^2 - \norm{x}_2^2)$ & [0;0;0] & [0.9; 0.2; 0.2] \\
        hollow ellipsoid  & $\sqrt{3}$ & 2 & $(x_1^2 + 2 x_2^2 + 3x_3^2- 0.6^2)(0.4^2 - x_1^2 - 2 x_2^2 - 3x_3^2)$ & [0;0;0] & [0.9; 0.2; 0.2]
    \end{tabular}
    \caption{Disconnectedness certificates of 3-dimensional systems}
    \label{tab:sys_3d}
\end{table}






%% file: sections/conclusion.tex
\section{Conclusion}
\label{sec:conclusion}


This work provided polynomial certificates of path-disconnectedness between given sets $X_0$ and $X_1$ inside a larger given set $X$. The connectiveness task is interpreted as a single-integrator optimal control problem that steers between $X_0$ and $X_1$. 
The path-disconnected certificates may be interpreted as time-dependent barrier functions that certify infeasibility of this optimal control problem within a specified time horizon. Upper-bounds on the true time horizon may be computed when $X$ is constructed from the union of simple sets, or when the set $X$ is described by a single polynomial inequality constraint. The control variables may be eliminated from the Lie constraint, thus improving the computational performance of \ac{SOS} barrier synthesis.

{Note also the recent work in \cite{korda2022urysohn} which provides a
polynomial certificate of non-intersection between two given semialgebraic sets $X_0$ and $X_1$. Similarly to what is done in our paper, this certificate is computed with the moment-SOS hierarchy. However, the certificate does not prove path-disconnectedness in $X$.}

This paper focused on the existential path-disconnectedness task, in which there does not exist $x_0 \in X_0, \ x_1 \in X_1$ such that $x_0$ may be connected to $x_1$ within $X$. Future work could involve the universal path-disconnected task, proving that there exists an $x_0 \in X_0, x_1 \in X_1$ such that $(x_0, x_1)$ are disconnected in $X$. Other future work includes improving the numerical conditioning of the \acp{SDP}, reducing the complexity of the Moment-\ac{SOS} programs, and obtaining necessary conditions for existence of  path-disconnectedness certificates in unbounded domains \cite{wu2023generalizing}.

%% file: sections/acknowledgements.tex
\section*{Acknowledgements}

The authors would like to thank Mario Sznaier, Fred Leve, Roy S. Smith, and the POP group at LAAS-CNRS for discussions about path-disconnectedness and occupation measures.

J. Miller was partially supported by NSF grants partially supported by NSF grants  CNS--1646121, ECCS--1808381 and CNS--2038493, AFOSR grant FA9550-19-1-0005, and ONR grant N00014-21-1-2431. J. Miller was in part supported by the Chateaubriand Fellowship of the Office for Science \& Technology of the Embassy of France in the United States and by the Swiss National Science Foundation Grant 51NF40\_180545.

D. Henrion's research stays in Paris were partly supported by the EOARD-AFOSR grant agreement FA8665-20-1-7029  coordinated by Mohab Safey El Din and Emmanuel Tr\'elat.

%% file: sections/app_strict.tex
\section{Strict Reformulation}
\label{app:strict}

This appendix shows how to replace nonstrict inequality constraints in \eqref{eq:altern_u} with strict inequality constraints, without affecting the feasibility of finding time-dependent barrier functions. 

\begin{lem}
\label{lem:altern_u_0}
    The initial constraint \eqref{eq:altern_u_init} may be replaced by a strict inequality to form the equivalent program 
\begin{subequations}
\label{eq:altern_u_0}
\begin{align}
   \find_v \quad &v(0, x) > 0 & & \forall x \in X_0 \label{eq:altern_u_0_init}\\
    &v(T, x) \leq 0 & & \forall x \in X_1 \label{eq:altern_u_0_term}\\
    &\Lie_u v(t, x) \geq 0 & & \forall (t, x, u) \in [0, T] \times X \times U \label{eq:altern_u_0_lie}\\
    &v(t, x) \in C^1([0, T] \times X).
\end{align}
\end{subequations}
\end{lem}
\begin{proof}
    Consider a feasible solution $v(t, x)$ of \eqref{eq:altern_u_0}. Given that $v\in C^1$ and $X_0$ is compact, the positive minimum $p^* = \min_{x \in X_0} v(0, x) > 0$ is attained. The positively-scaled barrier function $(1/p^*)v(t, x)$ therefore satisfies all constraints of \eqref{eq:altern_u}. Similarly, any solution of \eqref{eq:altern_u} satisfies the constraints of \eqref{eq:altern_u_0}, given that $\forall x \in X_0: v(0, x) \geq 1$ implies that $\forall x \in X_0: v(0, x) > 0$. Because the two feasibility sets are equal, the result follows.
\end{proof}

\begin{lem}
\label{lem:altern_u_0_strict}
        The Lie derivative constraints in \eqref{eq:altern_u_lie} and in \eqref{eq:altern_u_0_lie} may be replaced by strict inequalities to form the equivalent program 
        \begin{subequations}
\label{eq:altern_u_0_strict}
\begin{align}
  \find_{\tilde{v}}  \quad&\tilde{v}(0, x) > 0 & & \forall x \in X_0 \label{eq:altern_u_0_strict_init}\\
    &\tilde{v}(T, x) < 0 & & \forall x \in X_1 \label{eq:altern_u_0_strict_term}\\
    &\Lie_u v(t, x) > 0 & & \forall (t, x, u) \in [0, T] \times X \times U \label{eq:altern_u_0_strict_lie}\\
    &\tilde{v}(t, x) \in C^1([0, T] \times X). 
\end{align}
\end{subequations}
\end{lem}
\begin{proof}
    Let $v$ be a solution to \eqref{eq:altern_u_0} with a positive minimum  $p^* = \min_{x \in X_0} v(0, x) > 0$ in $X_0$ (as used in the proof of Lemma \ref{lem:altern_u_0}). We can define $\tilde{v}$ in terms of an $\epsilon>0$ with    
    \begin{subequations}
    \label{eq:tilde_v_subs}
        \begin{align}
            \tilde{v}(t, x) &= v(t, x) - (1-1/(2T))\epsilon \label{eq:tilde_v_subs_orig}\\
            \intertext{under the following relations:}
            \tilde{v}(0, x) &= v(0, x) - \epsilon \\
            \tilde{v}(T, x) &= v(T, x)  - \epsilon/2\\
            u \cdot \nabla_x \tilde{v}(t, x) &= u \cdot \nabla_x v(t, x) \\
            \partial_t \tilde{v}(t, x) &= \partial_t v(t, x) + (1/(2T))\epsilon.
        \end{align}
    \end{subequations}
    Substitutions of \eqref{eq:tilde_v_subs} into the left-hand-sides of \eqref{eq:altern_u_0_strict} yield
    \begin{subequations}
    \begin{align}
 \find_v \quad    &v(0, x) - \epsilon > 0 & & \forall x \in X_0 \label{eq:altern_u_0_init_subs}\\
    &v(T, x) -\epsilon/2 < 0 & & \forall x \in X_1 \label{eq:altern_u_0_strict_term_subs}\\
    &\Lie_u v(t, x) + \epsilon/(2T)  > 0 & & \forall (t, x, u) \in [0, T] \times X \times U \label{eq:altern_u_0_strict_lie_subs}\\
    &v(t, x) \in C^1([0, T] \times X).
    \end{align}
    \end{subequations}    
    Given that $v$ satisfies the nonstrict inequality constraints from \eqref{eq:altern_u}, choosing $\epsilon < p^*$ (such as $\epsilon = p^*/2$) allows for the $\tilde{v}$ from \eqref{eq:tilde_v_subs_orig} to satisfy the strict constraints in \eqref{eq:altern_u_0_strict}. The result follows.
\end{proof}


%% file: sections/app_poly_approx.tex
\section{Polynomial Approximation}
\label{app:poly_approx}
This appendix shows that $(v, \zeta)$ may be chosen to be polynomials when finding path-disconnectedness certificates.

\begin{lem}
\label{lem:c1_lie_u}
    Under assumptions A1 and A4, the $C^1$ norm of $v \in C^1([0, T] \times X)$ satisfies
    \begin{align}
        \norm{v}_{C^1([0, T] \times X)}  \geq  \norm{v}_{C^0([0, T] \times X)} + \norm{\Lie_u v}_{C^0([0, T] \times X \times [-1, 1]^n).}
    \end{align}
    \end{lem}
\begin{proof}
    
    The definition of the $C^1$ norm is
    \begin{align}
        \norm{v}_{C^1([0, T] \times X)} &= \norm{v}_{C^0([0, T] \times X)} + \norm{\partial_t v}_{C^0([0, T] \times X)} + \sum_{i=1}^n \norm{\partial_i v}_{C^0([0, T] \times X)}. \\
        \intertext{Given that $u_i \in [-1, 1]$ from assumption A4, the following ordering relations are obeyed}
        \norm{v}_{C^1([0, T] \times X)}&\geq \norm{v}_{C^0([0, T] \times X)}+ \norm{\partial_t v}_{C^0([0, T] \times X)}  + \sum_{i=1}^n \norm{u_i \partial_i v(t, x)}_{C^0([0, T] \times X \times [-1, 1]^n)}.  \\
        \intertext{The final relation holds given that $\norm{a+b} \leq \norm{a} + \norm{b}$ for all norms}
        &\geq \norm{v}_{C^0([0, T] \times X)} + \norm{\Lie_u v}_{C^0([0, T] \times X)}.
    \end{align}
    The result follows.
    \end{proof}

\begin{thm}
\label{thm:poly_v}
    Under assumptions A1-A3, there exists a polynomial $V(t, x) \in \R[t, x]$ that solves \eqref{eq:altern_u_0_strict} and certifies path-disconnectedness.
\end{thm}
\begin{proof}
This proof uses strategies from Theorem 2.3 of \cite{fantuzzi2020bounding} and Theorem 4.1 of \cite{miller2023robustcounterpart}.  Let $\tilde{v}$ be a solution to \eqref{eq:altern_u_0_strict}, and define positive tolerances $\eta, \delta > 0$. Theorem 1.1.2 of \cite{llavona1986approximation} may be applied to find a polynomial $w(t, x)$ such that $\norm{w(t,x)-\tilde{v}(t,x)}_{C^1([0, T] \times X)} \leq \eta$. 
By Lemma \ref{lem:c1_lie_u}, this approximation implies that $\forall (t, x) \in [0, T] \times X$:
\begin{subequations}
\begin{align}
    \tilde{v}(t, x) - \eta &\leq w(t, x) \leq \tilde{v}(t, x) +\eta \\
    \Lie_u\tilde{v}(t, x) - \eta & \leq \Lie_u w(t, x) \leq \Lie_u \tilde{v}(t, x) +\eta.
\end{align}
\end{subequations}
Define the polynomial
\begin{align}
    V(t, x) = w(t, x) - \delta(1-t/(2T)).
\end{align}
Similar substitutions to \eqref{eq:tilde_v_subs} may be performed to acquire
\begin{subequations}
    \begin{align}
        V(T, x) &= w(T, x) - \delta/2 & & \leq \tilde{v}(T, x) - \delta/2 + \eta \\
        V(0, x) &= w(0, x) - \delta & & \geq \tilde{v}(0, x)  - \delta - \eta  \\
        \Lie_u V(t, x) &= \Lie_u w(t, x) + \delta/(2T) & & \geq \Lie_u \tilde{v}(t, x) + \delta/(2T) - \eta .
    \end{align}
\end{subequations}
Satisfaction of the following constraints on $(\delta, \eta) > 0$  proves this theorem, certifying the existence of a polynomial $V$ that fulfills the requirements of \eqref{eq:altern_u_0_strict}:
\begin{align}    
\label{eq:eta_delta_cons}
     \eta + \delta & < \min_{x \in X_0} \tilde{v}(0, x) & \eta &< \min(1, 1/T) \delta/2.     
\end{align}
An admissible choice of $(\delta, \eta)$ that satisfies \eqref{eq:eta_delta_cons} is
\begin{align}
\eta^* &=  (\delta/4) \min(1, 1/T)) & \delta^* = \left[\min_{x \in X_0} \tilde{v}(0, x)\right] / (2 + \min(1, 1/T)/2),
\end{align}
thus proving the theorem.

\end{proof}

\begin{thm}
\label{eq:poly_zeta}
        For any valid  $\tilde{v}(t,x)$ satisfying \eqref{eq:altern_u_0_strict_lie_box}, the multipliers $\zeta$ from \eqref{eq:altern_u_box_con_zeta} may be chosen to be polynomial (under A1 and A4).
\end{thm}
\begin{proof}
    Polynomial approximability of $\zeta$ holds by  
 Theorem 4.3 of \cite{miller2023robustcounterpart} with respect to the set $U =  [-1, 1]^n$.
\end{proof}